\documentclass[10pt]{IEEEtran}

\IEEEoverridecommandlockouts                              
\newcommand{\ignore}[1]{}
\usepackage{dblfloatfix}
\usepackage{epsfig}
\usepackage{amsmath}
\usepackage{amsthm}
\usepackage{mathtools}
\usepackage{amsfonts}
\usepackage{amssymb}
\usepackage{boxedminipage}
\usepackage{wrapfig}
\usepackage{graphicx}
\usepackage{color}
\usepackage{multicol}
\usepackage{todonotes}

\usepackage[boxruled,vlined]{algorithm2e}
\SetAlFnt{\normalsize}

\newtheorem{lemma}{Lemma}
\newtheorem{theorem}{Theorem}

\newcommand{\Probs}{\mathbb P}

\newcommand{\Reals}{{\mathbb R}}

\newcommand{\tgi}{\lim_{t\rightarrow \infty}}

\newcommand{\e}{{\rm e}}

\newcommand{\pr}{{\mathcal P}}
\newcommand{\mI}{{\texttt{i}}}

\newcommand{\off}{{o\!f\!\!f}}

\newcommand{\Nm}{\mathbb{N}_m^+}
\newcommand{\Nn}{\mathbb{N}_n^+}

\newcommand{\one}{{\bf 1}}
\newcommand{\zero}{{\bf 0}}
\newcommand{\dg}{{\rm diag}}

\begin{document}
\title{\LARGE \bf
 Safe Markov Chains for Density Control of ON/OFF Agents with Observed Transitions}
\author{Nazl\i \  Demirer, Mahmoud El Chamie, and Beh\c cet\ A\c c\i kme\c se
\thanks{Authors are with the Department of Aeronautics and Astronautics, University of Washington, Seattle, WA  98195. Emails: \texttt{\{ndemirer, melchami, behcet\}@uw.edu }}}

\thispagestyle{empty}
\maketitle

\begin{abstract}
This paper presents a convex optimization approach to control the density distribution of autonomous mobile agents with two control modes: ON and OFF. 
The main new characteristic distinguishing  this model from standard Markov decision models is the existence of the ON control mode and its observed actions.  When an agent is in the ON mode, it can measure the instantaneous  outcome of one of the actions corresponding to the ON mode and decides whether it should take this action or not  based on this new observation. If it does {\em not} take this action, it switches  to the OFF mode where it transitions to the next state  based on a predetermined set of transitional probabilities, without making any additional observations.
In this decision-making model, each agent acts autonomously according to an ON/OFF decision policy, and the discrete probability distribution for the agent's state  evolves according to a discrete-time Markov chain that is  a linear function of the stochastic environment (i.e., transition probabilities) and the ON/OFF decision policy. The relevant policy synthesis  is formulated as a convex optimization problem where  hard  safety  and convergence constraints are imposed  on the resulting Markov matrix. 
 We first consider the case where the  ON mode has a single action and the OFF mode has deterministic transitions (rather than stochastic) to demonstrate the model and the ideas behind our approach, which is then generalized to the case where ON mode has multiple actions and OFF mode has stochastic transitions. 
\end{abstract}

\section{Introduction} \label{sec:intro}
This paper presents a convex optimization based  approach for the synthesis of  randomized decision policies to control the density   of mobile agents that switch between two modes: ON and OFF. Each mode consists of a (possibly overlapping) finite set of actions, that is, there exist a set of actions for the ON mode and another set for the OFF mode, which may have a non-empty intersection. At each time step, an agent is allowed to measure (observe) the instantaneous outcome, i.e. {\em transition}, for a single action it chooses among the set of actions for the ON mode (the actions that can be taken while in ON mode). Then, it decides whether to accept or reject this transition. 
Both decisions, i.e., selection of an action to observe and acceptance/rejection of corresponding transition are made based on a randomized decision policy. If the agent does not accept the proposed transition, it switches to the OFF mode where it transitions to the next state  based on a predetermined set of transitional probabilities, without making any additional observations. The probability  distribution  of each agent's state evolves according to the resulting, underlying,  finite-state and discrete-time Markov chain (MC). In the example  of large number of multiple, {\em swarm of}, autonomous agents,  the overall density distribution of the swarm can be controlled by designing the underlying MC via the ON/OFF decision policy.

 The control of systems with autonomous mobile agents has been a point of interest recently in many applications involving multi-agent sensor networks \cite{hill2000system}  such as micro vehicles for surveillance, search and rescue, mobile space-based  \cite{saaj2006spacecraft,ajc_13}, aerial  \cite{corke2004autonomous,wolf2010probabilistic}, naval \cite{heidemann2006research}, or land-based sensor networks \cite{Avrachenkov:2011}. The proposed Markov decision model and formulation  is applicable  for systems with both single and multiple agents, i.e., the density distribution  can be interpreted as the  temporal probability distribution of the state of a single agent, or  the state probability distribution over multiple agents. Among many possible applications of the theoretical framework presented, we  illustrate the results with a swarm control  example. Our previous research has developed methods for swarm control policy synthesis without the notion of agent modes or actions \cite{pga_acc12,ajc_13,auto2015}, then the idea is extended to control of a swarm of partially controlled ON/OFF agents in \cite{acc2015_onoff}, where we only considered the case when there is only a single action for the ON mode and a deterministic action for the OFF mode.  In this paper, we generalize this model for the case when there is an ON mode with  multiple  actions and an OFF mode with stochastic transitions, through a new Markov decision model with additional measurements for state transitions. 
 
 \begin{figure}[t]
\begin{center}
\includegraphics[width=3.3in]{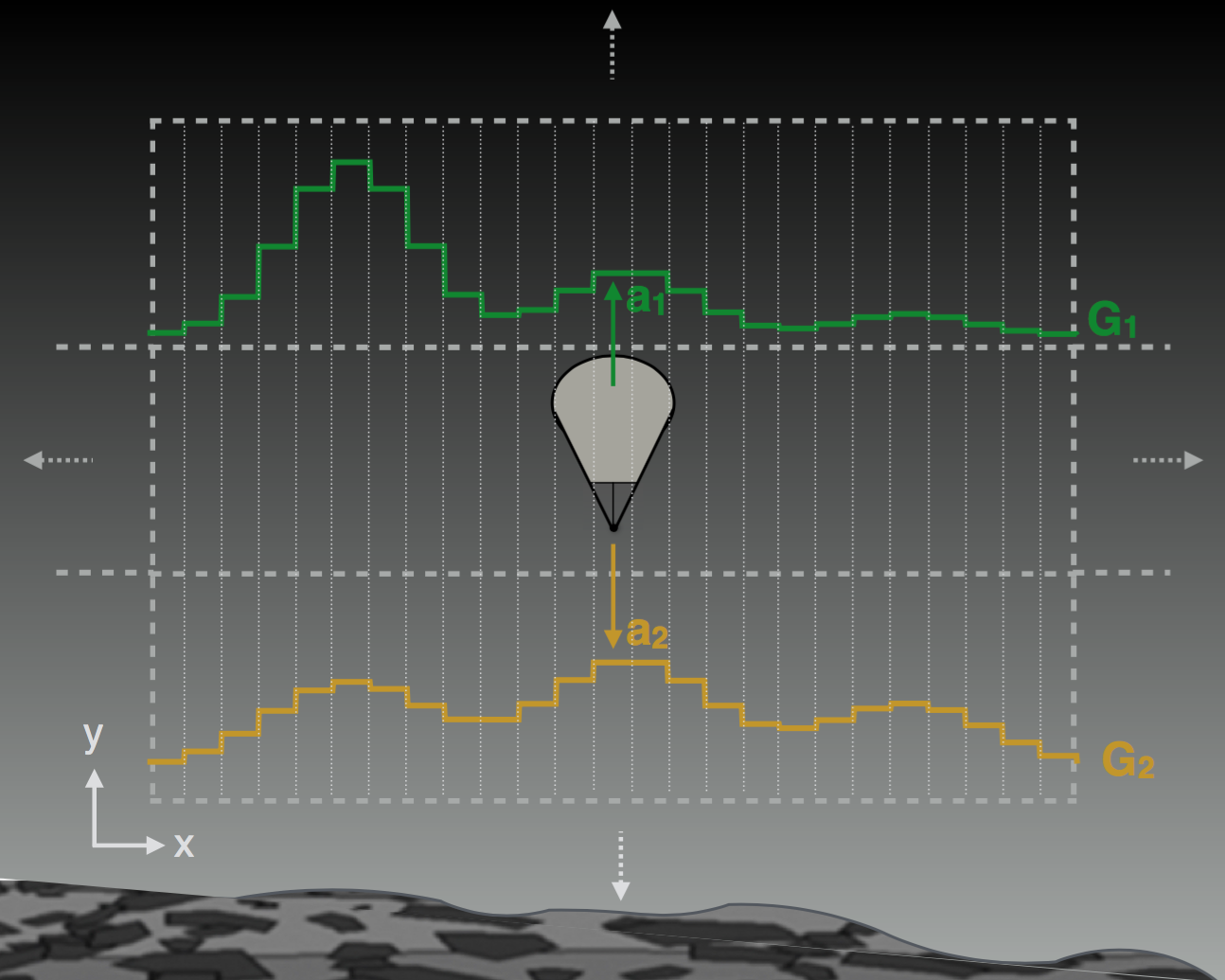}
\end{center}
\vspace{-3mm}
\caption{A simple 2D air balloon illustration with actions, $a_i$. $G_i$ is the discrete probability density distribution  for the $x$-position of the balloon resulting from taking action $a_i$ from its current position. In this example, the balloon observes outcome of an action by changing its altitude based on the action, and   the OFF mode is an MDP \cite{puterman14} running over all actions, hence $G_\text{off}$ is a function of $G_i$'s, i.e., $G_{\text{off}}~=~f(G_1,G_2,...)$.}
\label{balloon}
\end{figure}

 It is noteworthy that the OFF mode captures multiple interesting scenarios: (i) There is only a single action for the OFF mode for which the outcome is not observable; (ii) There are multiple actions for the OFF mode without outcome observations and a standard decision policy running over these actions, which results in an effective transition matrix for the underlying Markov chain; iii) Transitions can be observed for all available actions and there is a standard decision policy that can run over the actions without requiring transition observations, which is utilized as the default policy when the observed action  is rejected.  Note that in the third interpretation, the set of actions for the ON mode and the OFF mode are identical, and a rejected observed action at some instant may be still chosen when the default OFF mode decision policy  is executed after rejection.


For the formulation of the ON/OFF policy synthesis problem, we define two decision variables for each action: i) probability of choosing an action, i.e., a mapping from   states to a probability distribution over the finite number actions; ii) probability of accepting the corresponding transition with the chosen action. 
The policy synthesis problem turns out to be bilinear (hence non-convex) in these decision variables. However, applying a change of variable, we show that we can still obtain optimal solutions by solving an equivalent convex linear matrix inequality (LMI) optimization problem. When implementing this method, the decision policies  are computed  offline by solving LMI problem and given to each agent assuming that each agent can observe its current state, measure one-step outcome of a single action for the ON mode, and accept or reject a transition corresponding to the selected action at each time step. If this action is rejected, the agent switches to the OFF mode. In this sense ON and OFF modes can be seen as higher level actions, under which there are the lower level motion actions.
An interesting example to illustrate the concept of ON/OFF agents is controlling air balloons in uncertain wind fields \cite{wolf2010probabilistic} for scientific measurements in Earth and other planets \cite{kuwata2009decomposition,elfes2010implications}. Probabilistic  wind velocity field information, i.e., the probability density functions for the wind speed and direction,  changes as a function of the balloon's altitude. The hot air balloons can change their altitudes to choose the velocity field that they ride with, i.e.,  horizontal motion induced by the wind. Previous research proposed    controlled Markov process models for this example \cite{wolf2010probabilistic}. In our approach, the action in this Markov process model is the choice of altitude (see Figure \ref{balloon}).  Then we can synthesize a default Markov Decision Process (MDP) policy to distribute the balloons based on this prior wind field knowledge (transition probabilities as a function of the altitude), which defines  the OFF mode. We can also design an ON policy, for which we measure the instantaneous velocity observed at different altitudes. We can accept or reject the selected altitude (action) for the ON mode based on the instantaneous velocity observed at this altitude. If it is  rejected based on the observed outcome, we can go to the altitude suggested by the OFF mode policy. Note that, the sets of actions for the ON and OFF modes are the same in this example.   

\section{Related Research}
The proposed Markov model is applicable to both decision-making for single and multi-agent  systems  in stochastic environments. Our particular interest is motivated by 
the problem of guiding a multi-agent system, which has been a recent subject of research in Markov Decision Processes (MDPs) \cite{Sigaud:2010:MDP,El-Chamie:2016IJCAI,El-Chamie:2016SCMDP}. The problem presented here can also be considered as an  MDP \cite{puterman14} with finite number of states and actions, and a new set of observations. However, rather than having a reward function for each action and transition \cite{El-Chamie:2016SOMDP}, the mission objectives here are embedded within the underlying Markov chain.   

Most of the work in the area of MDPs and multi-agent systems focuses on the decentralized control of Markov decision processes (Dec-MDP) and decentralized partially-observed Markov decision processes (Dec-POMDP), e.g., \cite{amato2013decentralized}, where each agent has partial or incomplete observations of its state. These problems, generally, are very difficult to solve and become intractable for large scale problems \cite{bernstein2002}. These problems with incomplete observations  are quite interesting, but not directly comparable to the problem considered in this paper, where we utilize additional measurements for state-transitions. 

Using a Markov chain for guiding large numbers of agents  to a desired distribution is a relatively  new idea. \cite{berman_09} considers a similar problem in the task allocation framework where probability of switching between tasks are designed  to achieve maximum redistribution rate, without any additional constraints.  A different Markov chain based method is proposed  in \cite{ray_09} by
using a probabilistic disablement approach found in \cite{wonham_93}. \cite{hespanha_pg08} uses biased random walk, which leads vehicle positions to evolve toward a probability density. Another method for multi-agent coordination is to use nearest neighbor information to establish  consensus \cite{morse_03, bertsekas_07,sonia_07_tac,roy2015scaled,rahmani2009controllability}. \cite{cortez_04} proposes a method based on locational optimization and centroidal Voronoi diagrams, while \cite{pavone_11, bai2008adaptive}  provide adaptive methods of multi-agent  coordination.
Other stochastic approaches to multi-agent control include gradient-based decentralized controllers \cite{hsieh_08b} utilizing relative positions to neighbors; game theoretical formulations \cite{arslan_07}, where each vehicle is considered as a self-interested player; maze searching techniques \cite{lumelsky_97} and cyclic pursuit strategies \cite{pavone_07}, where each agent pursues its leading neighbor  resulting  in convergence to a prescribed geometric pattern. In \cite{hsieh_08a}, a quorum based method is proposed where the multi-agent system is modeled as a hybrid system.  Reference \cite{sonia_07} presents a comprehensive survey of recently developed theoretical tools for modeling, analysis, and design of motion coordination algorithms. An interesting application introduced in the previous section, which can also be explored in the proposed ON/OFF framework, is the control of hot air balloons in stochastic wind fields for atmospheric science observations \cite{elfes2010implications}. Earlier work converted this motion planning problem into a more standard Markov decision model \cite{kuwata2009decomposition,wolf2010probabilistic}.
The main distinction of our paper from the listed references above is the existence of the ON control mode and its observed actions. This allows us to devise new methods to control the density distribution of autonomous agents via a new Markov decision model with measurements on the state transitions. Measurements for the ON mode can be obtained by the deployment of additional sensors to extend the agents' sensing capabilities. In summary, the key contributions of this paper are: i) Formulation of a new Markov chain synthesis problem through a new Markov decision model, with additional measurements for the state transitions, where a policy is designed to ensure that the desired safety and convergence properties for the underlying Markov chain; ii) Convexification of the synthesis problem; iii) Application of the model to density control of swarm of autonomous mobile agents.

The rest of the paper is organized as follows: Section~\ref{sec:Markovmodel} summarizes the formulation of the density control problem for ON/OFF agents for the case where ON mode has a single action and OFF mode has deterministic transitions; Section~\ref{sec:multipleOn} generalizes the problem to having multiple actions for the ON mode and stochastic transitions for the OFF mode, provides the algorithm for implementation and makes connections to MDPs; Section~\ref{sec:ConvexSynthesis} describes the ergodicity, transition and safety constraints and formulates the convex LMI problem with this constraints and Section~\ref{sec:simulation} has an illustrative example which uses the Markov matrix with the safety constraints. Section~\ref{sec:conclusion} concludes the paper.

\subsection*{Notation}
The following is a partial list of notation used:  $\zero$ is the zero vector/matrix of appropriate 
dimensions; 
 $I$ is the identity matrix; $\one$ is the vector of ones with appropriate dimensions; 
  $\e_i$ is a vector of appropriate dimension  with its $i^{th}$ entry $+1$ and its other entries 
zeros;  $x[i]= \e_i^T x$ for any $x \in \Reals^n$, and $A[i,j]=\e_i^T A \e_j$ for any 
$A \in \Reals^{n \times m}$;
$Q=Q^T \succ (\succeq) 0$ implies that $Q$ is a symmetric positive (semi-)definite matrix; 
$R > (\geq) H$ implies that  $R[i,j]> (\geq) H[i,j]$ for all $i,j$;  $R> (\geq) \zero$ implies that $R$ a 
positive (non-negative) matrix;  $x \in \Probs^{n}$ is said to be a {\em probability vector} if $x \geq \zero$ and $\one^{T} x = 1$; matrix $M \in \Probs^{m\times m}$ is a Markov matrix if $M \geq \zero$ and $\one^T M=\one^T$;
$\pr$ denotes probability of a random variable;
$\Reals^n$ is the $n$ dimensional real vector space;
$\mathbb{N}$ is set of nonnegative integers, i.e., $\mathbb{N}=\{0,1,2,\ldots\}$;
$\mathbb{N}^+_n=\{1,2,\ldots n\}$;
$\emptyset$ denotes the empty set;
 $(v_1,v_2,...,v_n)$ represents a vector obtained by augmenting vectors $v_1,\ldots,v_n$ {such that
$
(v_1,v_2,...,v_n) \equiv 
\left[v_1^T \   \  v_2^T\     \ldots  \ v_n^T \right]^T
$
where $v_i$ can have arbitrary dimensions;  $\dg(A)= (A[1,1],\ldots,A[n,n])$ for matrix $A$; $\dg(v)$ is a square diagonal matrix with the elements of vector $v$ on the main diagonals;
$\otimes$ denotes the Kronecker product; 
$\odot$ represents the Hadamard (Schur) product; $\mI(A)$ is the indicator matrix for any matrix $A$, whose 
entries are given by $\mI( A)[i,j]  = 1$ if $ A [i,j] \neq 0$ and $\mI(A)[i,j] =0$ otherwise. $\eta \sim U(0,1)$ denotes a random variable sampled from the uniform distribution in the interval $[0,1]$. 
%

\section{Markov Chain Model for Density Control of  ON/OFF Agents: Single ON  action case}\label{sec:Markovmodel}
In this section, we introduce the probabilistic density control problem for ON/OFF agents. First, we consider the simplified model of mode-switching ON/OFF agents to demonstrate the concept and formulation that  first appeared in \cite{acc2015_onoff}. Then, we generalize this model in the next section. In both general and simplified models, agents are assumed to know their current states. For the simplified model, we consider the case where the ON mode contains only one action for which  the action outcome is observable, i.e., agents can measure their next state at a given instant if they take the action (i.e., they decide to be in the ON mode). However agents  do not have this  transitional   observation  for the single action in the OFF mode. For ease of demonstration of the idea, we assume that the stochastic transition matrix for the action in OFF mode is an identity matrix. Note that having an identity matrix  means that the outcomes of the action for the OFF mode are deterministic (they do not require measurement of the transitions), which is a trivial case and it will be generalized much further in the next section. This assumption makes the initial formulation more transparent. Furthermore, this  simpler model also captures an interesting interpretation of the density control problem as presented in \cite{acc2015_onoff}, that is, ON/OFF agents are moving with the dynamics induced by the stochastic environment when they are ON and they are staying stationary when they are OFF. 

 The density control problem for this simplified model can be formulated  as a Markov chain synthesis problem. For that, we define a finite set of states   $\mathcal{S} = \{ s_1,\ldots,s_n \} $, that is, there is  a discrete state space with cardinality $n$ and $s_j$ is referred to as ``$j^{th}$ state" in remainder of the paper.  We consider a discrete-time system where $s(t) \in \mathcal{S}$ is the state of the agent at time epoch $t$, i.e., $s(t)=s_i$ is the event that the state is the $i^{th}$ state at time $t$.
Then, probability density distribution $x(t) \in \mathbb{P}^n $ is defined as:
\begin{equation} \label{eq:pdf}
x[i](t) = \pr \{ s(t) = s_i \} , \qquad i \in \mathbb{N}^{+}_n, \ \ t \in \mathbb{N},
\end{equation}
where $t$ is the discrete time index. 
Hence $x[i](t)$ is the probability of a mobile agent to be in the $i^{th}$ state at time~$t$.
In the rest of the section, we present the  formulation for  the time evolution of the density $x(t)$ as the following Markov chain, which is  
  defined over the state-space $\mathcal{S}$:
\begin{equation} \label{eq:markov}
x(t+1)=Mx(t),
\end{equation}
where $M$ is the transition matrix, i.e., $M[i,j]=\pr \{s(t+1)=s_i | s(t)=s_j\}$.  The transition matrix $M$ is a function of the stochastic environment and the  ON/OFF policy, as will be explained  next. 
We define the following events to properly define the stochastic environment and the decision policy, for $t \in \mathbb{N}$:
\begin{align*}
y(t+1)=s_l & \ \text{: Observing a transition to state $s_l$} \\
\sigma(t)=\text{$\sigma_{on}$} & \ \text{: Accepting to execute the action for the}\\
& \hspace{4mm} \text{ON mode}.
\end{align*}
Note that we have two {\em modes} of operation $\sigma(t)~ \in~ \{ \sigma_{on},\sigma_{\off}\}$.
Even though $y(t+1)$ has a time index $t+1$, this observation occurs at time $t$. In particular, observation of a transition is different from the actual transition taking place. For example, $y(t+1)~=~s_l$ is the  event that the stochastic environment would have caused  a transition to $l^{th}$ state at time $t+1$ if  the action for the ON mode were to be accepted at time $t$ (i.e., observing one-step ahead in the future), whereas $s(t+1)~=~s_l$ is the event  that the transition to $l^{th}$ state has actually occurred. 

The stochastic environment is defined with the transition matrix $G\in \mathbb{P}^{n\times n}$, where $G[i,j]$ is the probability of observing a transition from $j^{th}$ state to $i^{th}$ state when the action for the ON mode is taken, i.e.,
\begin{equation} \label{eq:G_def}
G[i,j]\!\!:= \pr\{y(t\!+\!1)=s_i| s(t)=s_j\}.\end{equation}
%
In this paper, the environment transition matrices and the decision policy are assumed to be time-invariant (i.e., the processes are stationary), hence the corresponding Markov chain transition matrix given in (\ref{eq:markov}) is also time-invariant. Our objective is to synthesize  a decision policy for an agent to accept or reject the corresponding transition observed for the action in the ON mode at each time epoch, i.e. to decide whether it should be ON or OFF,  such that the resulting Markov chain will satisfy the desired transition and safety constraints while guiding the density distribution to a desired final distribution. In this section, the state of the agent is assumed  not to change if it is OFF, i.e., $s(t+1)=s(t)$ when $\sigma(t)\!=\!\sigma_{\off}$.   Then the probabilistic ON/OFF  decision policy is defined by a matrix $K \!\in\! \Reals^{n\times n}$ (to be designed) that satisfies:
\begin{equation}
\zero \leq K \leq \one \one^T, \quad \dg(K) =\one, 
\end{equation}
where $K[i,j]$ is the probability of being ON, given that the transition from $j^{th}$ state to $i^{th}$ state is observed, i.e., the acceptance probability of the environmentally induced motion determined by G: 
\begin{equation} \label{eq:K_def}
K[i,j] = \pr\{\sigma(t) = {\sigma_{on}}|s(t)=s_j, y(t+1)=s_i\}.
\end{equation}
Note that for the transition $i \rightarrow i, \ \forall i$, accepting or rejecting the transition corresponds to the same outcome in this case (in  this section). Hence, the diagonal elements of acceptance matrix $K$ are set to $1$.
We can now give the resulting ON/OFF decision-making policy for an agent in Algorithm~\ref{alg1}.
\IncMargin{0.5em}
\begin{algorithm}\label{alg1}
 \SetAlgoLined 
 \LinesNumbered
\SetKwInOut{Input}{Inputs}
\Input{$K$ (matrix designed offline), $\mathcal{S}, t_{\rm max}$}
 \For{$t\leftarrow 1$ \KwTo $t_{\rm max}$}{
    Observe the current state,   $s(t)\in \mathcal{S}$ (assume $s(t)=s_j$) \;
    Generate a random number $\eta(t) \sim U(0,1)$\;
        Observe the  transition outcome if the agent is in the ON mode, i.e.,  $y(t+1)$ (assume $y(t+1)=s_i$)\;
      \eIf{$\eta(t) \in [0,K[i,j]]$}{The agent switches to the mode ON $\sigma(t) = {\sigma_{on}}$, and $s(t+1)=y(t+1)$\;} {
      The agent switches to the mode OFF $\sigma(t) = {\sigma_{\off}}$, and $s(t+1)=s(t)$\;}}
     \caption{ON/OFF Decision-Making Policy -- Single action case}
\end{algorithm}

 Suppose that $M$ is the effective Markov matrix for the system when the  ON/OFF policy in Algorithm~\ref{alg1} is active. The next theorem shows that $M$ is a linear function of the design matrix $K$. This linearity property will be used for a convex synthesis of the algorithm design matrix $K$ so that the matrix $M$ satisfies some favorable properties as convergence and safety as we will show later on in the paper.  

\begin{theorem}  \label{lem:MKdyn} 
Consider a system of single or multiple ON/OFF agents moving in a stochastic environment defined by a finite number of states $\mathcal{S}$ with the transition probabilities given by $G$ as in (\ref{eq:G_def}) for the ON mode and with no transitions for the OFF mode.  Suppose that each agent executes  the ON/OFF decision-making Algorithm~\ref{alg1}
with  the matrix $K$  defined in (\ref{eq:K_def}). Then the density distribution $x(t)$ defined in (\ref{eq:pdf}) evolves based on a  Markov chain as in (\ref{eq:markov}), where
the Markov matrix $M \in \Probs^{n \times n}$ is given as follows: For  $ i,j \in \mathbb{N}^{+}_n$,\vspace{-2mm}
\begin{equation}\label{eq:GK}
M[i,j]=
\begin{cases}
  G[i,j] \, K[i,j]  & \text{ if } \quad i \neq j,  \\
   1 -   \sum\limits_{\substack{ k=1 \\k \neq j}}^n G[k,j] K[k,j] & \text{ if } \quad i = j.
\end{cases}
\end{equation}
%
In matrix form, the above relationship is equivalent to \vspace{-2mm}
\begin{equation} \label{eq:M_cl}
M = G \odot K + \dg (\one^T -\one^T (G \odot K)).
\end{equation}
\end{theorem}
\begin{proof}
Transition from the $j^{th}$ state to the $i^{th}$ state ($i\neq j$) can only take place when that transition is observed (line~$4$ in Algorithm~\ref{alg1}) and accepted (line~$6$ in Algorithm~\ref{alg1}):
\begin{equation}
\begin{split}
M[i,j] &= \pr\{s(t+1)=s_i|s(t)=s_j\} \\
&=\pr\{\sigma(t) = {\sigma_{on}},y(t+1)=s_i|s(t)=s_j\} 
 \end{split}
\end{equation}
Using Bayes' Rule, we have:
\begin{equation}
\begin{split}
M[i,j]&=\pr\{y(t\!+\!1)=s_i| s(t)=s_j\}\times \\
& \pr\{\sigma(t) = {\sigma_{on}}|s(t)=s_j, y(t+1)=s_i\} \\
&= G[i,j]K[i,j]
 \end{split}
\end{equation}
The system stays at the $j^{th}$ state if either (a) the observed state at line~$4$ in Algorithm~\ref{alg1} is rejected (line~$8$ in Algorithm~\ref{alg1}) or (b) the $j^{th}$ state is observed at  line~$4$ in Algorithm~\ref{alg1} and is accepted (line~$6$ in Algorithm~\ref{alg1}), then \vspace{-2mm}
\begin{align*}
M[j,j]&=\pr\{s(t\!+\!1)=s_j| s(t)=s_j\}  \\
&\hspace*{-0.7cm}=\pr\{\sigma(t) = {\sigma_{\off}}|s(t)=s_j\} \\
&\hspace*{0.5cm}+\pr\{\sigma(t) = {\sigma_{on}},y(t+1)=s_j|s(t)=s_j\}\\
&\hspace*{-0.7cm}=1-\textstyle{\sum\limits_{k=1}^n\pr\{\sigma(t) = {\sigma_{on}},y(t+1)=s_k|s(t)=s_j\} }\\
&\hspace*{0.5cm}+\pr\{\sigma(t) = {\sigma_{on}},y(t+1)=s_j|s(t)=s_j\}\\
&\hspace*{-0.7cm}=1 -  \textstyle{\sum\limits_{\substack{k=1 \\ {k \neq j}}}^n G[k,j] K[k,j]}
\end{align*}

Note that we can verify that $M  \in \Probs^{n \times n}$ by showing that $M \!\geq\! 0$ and $\one^T M \!=\! \one^T$.
\paragraph*{$M \!\geq\! 0$}  Letting $S\!:=\!G\!\odot \! K$,  since $0\!\leq \!G, \, K \!\leq\! \one \one^T $, we have  $\one^T S \!=\! \one^T (G\odot K) \leq \one^T G\!=\! \one^T.$ Hence $\one^T\! \!-\one^T S \!\geq\! \zero$, which implies that $M\!=\!S\!+\dg(\one^T \!-\! \one^T S ) \!\geq\! \zero$.
 \paragraph*{$\one^T M \!=\! \one^T$} Let $\xi^T \!:=\! \one^T G\odot K$. Then $\one^T M \!=\! \xi^T \!+\! \one^T \dg (\one^T \!-\! \xi^T)= \xi^T  \!+\! \one^T (I\!-\!\dg (\xi) ) = \xi^T \!+\! \one^T \!-\! \one^T \dg (\xi) = \xi^T\!+\! \one^T \!-\!\xi^T =\one^T$.
\end{proof}
{\bf Remark:} The formulation for Markov matrix $M$ given in (\ref{eq:GK}) is quite intuitive. The probability of making transition $i\rightarrow j$ is simply the environment  induced probability of this transition times the probability of accepting this transition. The diagonal entries of $M$ are set so that the resulting Markov matrix  satisfies the column stochasticity property.$\hfill \blacksquare$



\section{Generalization of Density Control Problem for Multiple ON Actions}\label{sec:multipleOn}
In this section, we present the main result of this paper which is the generalization of the ON/OFF decision control policy problem for the case where the ON mode encapsulates multiple actions and the OFF mode features a single action that does not necessarily correspond to ``no motion". As explained earlier: in the ON mode, the ``next step" outcomes of the actions are observable, while a Markov chain, $G_{\off}$, is propagated when OFF mode is chosen.  
Now, we can have multiple actions in the ON mode whose transitions can be observed, i.e., 
\begin{equation} \label{eq:multi_ON}
\sigma(t) = \sigma_{on}    \ \implies \ a(t) \in  \mathcal{A}_{on} =\{ a_1,\ldots,a_m\},
\end{equation}
where $a(t)$ is the action taken in the ON mode. In the simple ON/OFF case of the previous section, we had $\mathcal{A}_{on}=\{a_1\}$, hence we did not need to explicitly define actions. Since we have multiple actions in the ON mode now, we have to explicitly identify them. 

Using the same definitions for $x(t)$ and $M$ given by (\ref{eq:pdf}) and (\ref{eq:markov}), we will formulate the Markov matrix $M$ as a function of the stochastic environment and the actions determined by a predetermined ON/OFF policy. For this general case, we expand the definitions of the probabilistic events for $t \in \mathbb{N}$: 
\begin{align*}
y(t+1)=s_l  \ &\text{: Observing a transition to state $s_l$} ,\\
v(t)= a_{k}   \ &\text{: Observing the outcome of taking action} \\
&\quad \text{$a_{k}$ from $\mathcal{A}_{on}$} ,\\
a(t)=a_k  \ &\text{: Accepting to execute action $a_k$}.
\end{align*}
Comparing to the previous model with single action, the addition here is the event ``$v(t)= a_{k}$", which is used to define the probability of choosing an action whose outcome will be observed.  
The stochastic environment is defined with the transition matrices $G_k, G_{\off} \in \mathbb{P}^{n \times n}, \ k \in \mathbb{N}^{+}_m$ where $G_k[i,j]$ gives the probability of observing a transition from $j^{th}$ state to $i^{th}$ state, given that the $k^{th}$ action is selected to be observed, i.e.,
\begin{equation} \label{Gk_def}
G_k[i,j]  = \pr \{ y(t+1) =s_i| s(t)=s_j, \, v(t) = a_k \}, 
\end{equation}
where $i,j  \in \mathbb{N}^{+}_n, \ \ k \in  \mathbb{N}^{+}_m$
and similarly $G_{\off}[i,j]$ defines the corresponding transition probabilities for the action in the OFF mode (e.g., $G_{\off} =I$ when being OFF means  no motion as in the previous section). 
The model for ON/OFF decision-making has the following assumptions (see Fig.~\ref{multiaction}):
\begin{itemize} 
\item Each agent measures its own  state at time instant $t$.
\item Agent chooses a single action for the ON mode, say $a_k$, whose outcome will be observed, i.e.,  $v(t)=a_k$.
\item Agent accepts or rejects to take the observed action.
\item If action is accepted then it is taken, i.e., $a(t)=a_k$, and transitions occurs according to $G_k$. 
\item If action is rejected, the agent chooses the OFF mode, $\sigma(t)=\sigma_{\off}$ and the transition occurs according to $G_{\off}$.
\end{itemize}
\begin{figure}[h!]
\begin{center}
\includegraphics[width=3.2in]{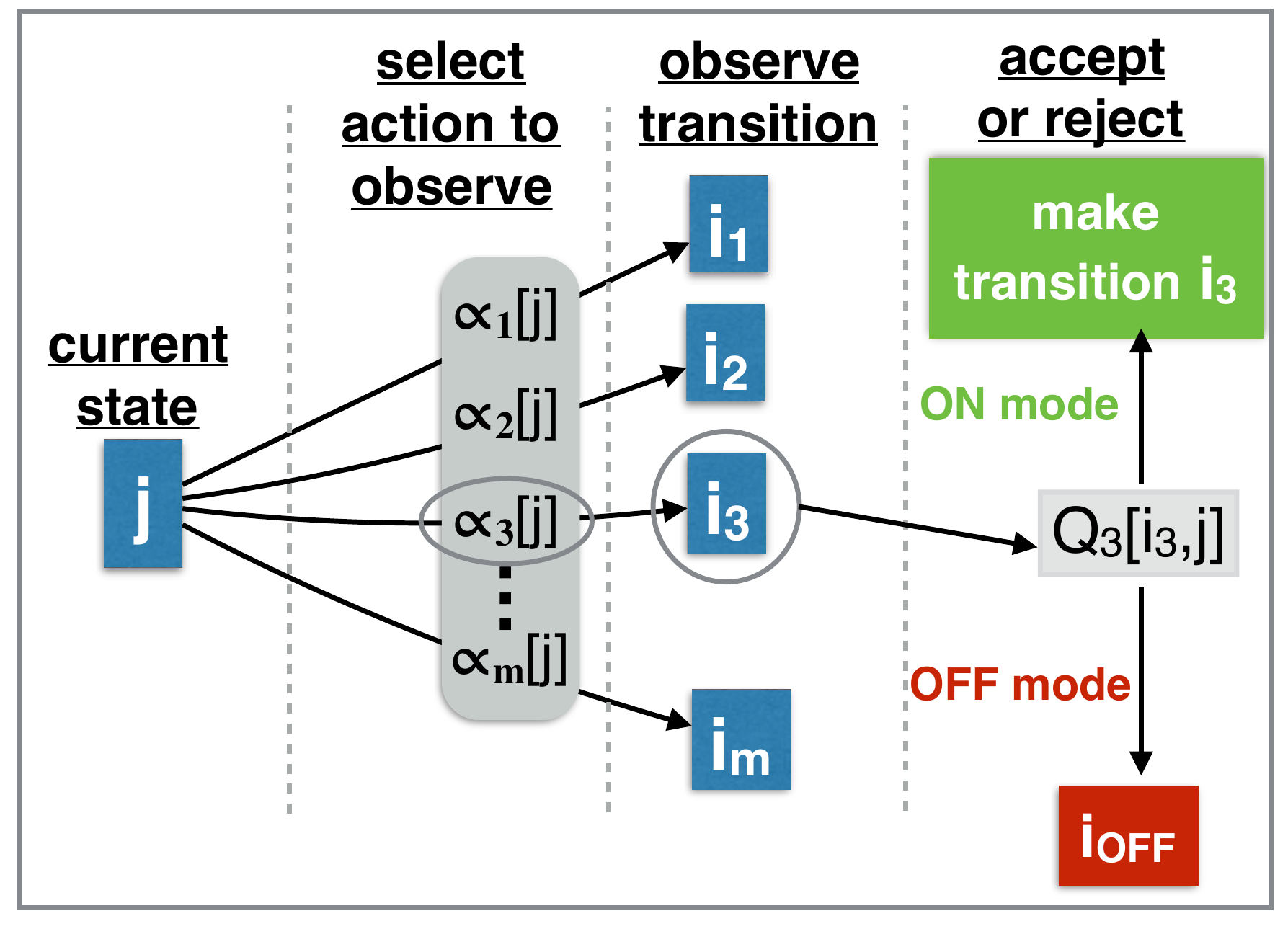}
\end{center}
\vspace{-5mm}
\caption{Implementation of the decision policy.}
\label{multiaction}
\end{figure}
{\bf Remark:} In the above setting, $G_{\off}$ can have two interpretations: (i) the environment transition matrix when there is only one action for which the outcome is not observable; (ii) the effective transition matrix when there are multiple actions with a prescribed decision policy. For the latter case, $G_{\off}$ is a Markov matrix  defining the underlying Markov chain of an MDP with an existing policy running over the actions, which is discussed in more detail in the next subsection.
$\hfill \blacksquare$
\\[2pt]
In the generalized model, we consider two sets of decision variables (to be designed offline): 
\begin{eqnarray}
\alpha _k[j] &= &\pr\{v(t)=a_k| s(t)=s_j\} \hspace*{2.5cm}\label{eq:alpha}\\
Q_k[i,j] & = &\pr\{a(t)=a_k|s(t)=s_j, v(t)=a_k, \nonumber \\
  && \hspace{3.5cm}y(t+1)=s_i\} \label{eq:Qk}
\end{eqnarray}
with $k \in \mathbb{N}^{+}_m, \ i,j \in \mathbb{N}^{+}_n$. Namely, $\alpha _k[j]$ is the probability of choosing action $a_k\in \mathcal{A}_{on}$ at state $s_j$ and $Q_k[i,j]$ is the probability 
of accepting an achievable transition $j\rightarrow i$ observed as an outcome of taking an action $a_k$.

Clearly $Q_k$ matrices must be non-negative, $Q_k~\in~[0,1]^{n\times n},\ k \in  \mathbb{N}^{+}_m$. Also, non-negative action variables $\alpha _k$ should satisfy the inequality $\sum_{k=1}^m\alpha _k[j]\leq 1,\  j\in \mathbb{N}^{+}_n$.
It turns out in our model (which will become more clear later) that we can combine these two variables by considering the change of variables
\begin{equation} \label{eq:Pk}
P_k := Q_k \dg (\alpha_k) = Q_k \odot (\one \alpha_k^T), \quad  k \in \mathbb{N}_m^+.
\end{equation}
The following property holds for $P_k$ matrices,
\begin{equation}
P_k  \leq \one \, \alpha_k^T  , \ \ \ k \in  \mathbb{N}^{+}_m.
\end{equation} 
This inequality can simply be proven by contradiction. Suppose there exist $i$ and $j$ such that $P_k[i,j]>\alpha _k[j]$, then $Q_k[i,j]=P_k[i,j]/\alpha _k[j] >1$ which is a contradiction because $Q_k[i,j]\in [0,1]$.

We can now give by Algorithm~\ref{alg2} the ON/OFF decision-making policy for the general case. For this algorithm, we define variable $\phi \in \Reals^{m+2}$ where
$$
\phi[1] = 0,\ \phi[r] = \sum_{l=1}^{r}\alpha_l, \ \phi[m+2]=1
$$
where $r = 2,\ldots, m+1$. 
\begin{algorithm}[h!]  \label{alg2}
	\SetAlgoLined
	\LinesNumbered
\SetKwInOut{Input}{Inputs}
\Input{$\{\alpha _k,Q_k: k \in \mathbb{N}^{+}_m\}$ (designed offline), $\mathcal{S}, t_{\rm max}$}
   \For{$t\leftarrow 1$ \KwTo $t_{max}$}{
		Determine current state $s(t)\in  \mathcal{S}$ (assume $s(t)=s_j$)\;
		Generate a random numbers $\mu(t) \sim U(0,1)$ and $\eta(t) \sim U(0,1)$\;
				\eIf{$\mu(t)\in [\phi[k] , \   \phi[k+1] )$}{
			$v(t)=a_k$\;
			Observe the next achievable transition for $a_k$: $y(t+1)$ (suppose
			$y(t+1) = s_i$)\;
				\If{$\eta(t) \in [0,Q_k[i,j]]$}{The agent switches to the mode ON, $a(t) = a_k$ and $s(t+1)=s_i$ \;}
		   } 
		   {
			The agent switches to the mode OFF $a(t) = a_\off$, and $s(t+1)$ transitions according to $G_{\off}$\;
		}
	}
	\caption{ON/OFF Decision-Making Policy -- General case}
\end{algorithm}

The following theorem presents the key  result in converting the ON/OFF decision policy  design problem  into a Markov chain synthesis problem.
\begin{theorem} \label{lem:Mdyn}
Consider a system of single or multiple mode switching ON/OFF agents moving in a stochastic environment defined by a finite number of states $\mathcal{S}$ with the transition probabilities given by $G_k$ as in (\ref{Gk_def}) for the $k^{th}$ action $a_k \in \mathcal{A}_{on}$ in the ON mode and $G_{\off}$ for the OFF mode. Suppose that each agent executes  the ON/OFF decision-making Algorithm \ref{alg2}
with the matrices $Q_k$  as in (\ref{eq:Qk}) and the vectors $\alpha_k$ as in (\ref{eq:alpha}). Then the state p.d.f. $x(t)$, defined in (\ref{eq:pdf}), evolves based on the  Markov chain   (\ref{eq:markov}) with the Markov matrix $M \in \Probs^{n \times n}$  given by 
 \begin{equation}\label{eq:Mmat}
M \!=\! \sum_{k=1}^m  G_k  \odot P_k  \, + \,  G_{\off} \odot \left(\one \left(\one^T \!-\! \one^T   \sum_{k=1}^m  G_k \! \odot\! P_k \right) \right) ,
\end{equation}
where $P_k, \, k \in \mathbb{N}^+_m$ are given by (\ref{eq:Pk}) and they satisfy
\begin{equation}\label{eq:Pk_sat}
\sum_{k=1}^m {\max_{i\in  \mathbb{N}^{+}_n} P_k[i,j]} \leq 1, \quad j \in \mathbb{N}^{+}_n.
\end{equation}
\end{theorem}
\begin{proof} The probability of making transition from $j^{th}$ state to $i^{th}$ state can be written as the following sum, since being ON and being OFF are mutually exclusive events:
\begin{equation}\label{eq:markov_sum}
\begin{split}
M[i,j] =  &\pr\{s(t\!+\!1)=s_i|s(t)=s_j\} \\
= &\underbrace{ \pr \{ \sigma(t)= \sigma_{on},\, s(t\!+\!1)=s_i|s(t)=s_j\}}_{:=T_1}\\
+&\underbrace{\pr\{\sigma(t)=\sigma_{\off},\, s(t\!+\!1)=s_i|s(t)=s_j\}}_{:=T_2}.
\end{split}
\end{equation}
In Algorithm \ref{alg2}, since executing the ON mode implies that an action from $\mathcal{A}_{on}$ must be taken, the first term can be written as the summation over all actions for the ON mode:
\begin{eqnarray*}
T_1 &=& \sum_{k=1}^m {\pr\{s(t\!+\!1)\!=\!s_i,a(t)\!=\!a_k|s(t)\!=\!s_j\}} \\
&=& \sum_{k=1}^m {\pr\{s(t\!+\!1)\!=\!s_i,a(t)\!=\!a_k,} \\
&&\hspace{1cm} v(t)\!=\!a_k|s(t)\!=\!s_j\}
\end{eqnarray*}
The second equation above follows from the fact that $v(t)$ should always precedes $a(t)$, i.e.,  $\pr(v(t)=a_k|a(t)=a_k)=1$. 
By applying Bayes' rule \cite{chung} to the term inside the sum, we obtain:
\begin{equation*}
\begin{split}
&T_1= \sum_{k=1}^m {\pr \{s(t\!+\!1)\!=\!s_i,a(t)\!=\!a_k|v(t)\!=\!a_k, } \\
&\hspace{2cm}  s(t)\!=\!s_j \} \underbrace{\pr\{v(t)\!=\!a_k|s(t)\!=\!s_j\}}_{\displaystyle \alpha_k[j]}.
\end{split}
\end{equation*}
Since observing transitions to distinct states are mutually exclusive events,
\begin{eqnarray}
&T_1=\sum\limits_{k=1}^m \bigg[\sum\limits_{l=1}^n {\pr\{s(t\!+\!1)\!=\!s_i,a(t)\!=\!a_k,} \nonumber\\
&\hspace{2cm} y(t\!+\!1)\!=\!s_l| v(t)\!=\!a_k,s(t)\!=\!s_j\}\bigg] \alpha_k[j] \nonumber
\end{eqnarray}
Applying Bayes' rule, we obtain:
\begin{eqnarray}
&=\sum\limits_{k=1}^m \bigg[\sum\limits_{l=1}^n {\pr\{s(t\!+\!1)\!=\!s_i,a(t)\!=\!a_k|} \nonumber\\
&\hspace{1cm}y(t\!+\!1)\!=\!s_l,v(t)\!=\!a_k,s(t)\!=\!s_j\} \nonumber\\
&\hspace{1cm} \times \underbrace{\pr\{y(t\!+\!1)\!=\!s_l|v(t)\!=\!a_k,s(t)\!=\!s_j\}}_{G_k[l,j]}\bigg]\alpha_k[j] \nonumber\\
&=\sum\limits_{k=1}^m\bigg[\sum\limits_{l=1}^n {\pr\{s(t\!+\!1)\!=\!s_i|y(t\!+\!1)\!=\!s_l,} \nonumber\\
& \hspace{1cm}a(t)\!=\!a_k, v(t)\!=\!a_k,s(t)\!=\!s_j\} \nonumber\\
&\hspace{1cm}\times \underbrace{\pr\{a(t) \!=\! a_k| y(t\!+\!1)\!=\!s_l,v(t)\!=\!a_k,s(t)\!=\!s_j\}}_{Q_k[l,j]} \nonumber\\
&\hspace{1cm} \times G_k[l,j]\bigg]\alpha_k[j]. \label{eq:temp1}
\end{eqnarray}
Note that, if the outcome of  $a_k$ is observed ($v(t)~=~a_k, \, y(t+1)~=~s_l$) and this action is accepted ($a(t)=a_k$), then the observed transition takes place with probability 1. In this case, a transition to any other state can take place only if the OFF mode is selected, i.e., $\sigma(t)=\sigma_{\off}$. Hence, we have 
\begin{equation} \label{eq:temp2}
\begin{split}
\pr\{s(t+1)&=s_i|y(t+1)=s_l,a(t)=a_k, \\
&v(t)=a_k,s(t)=s_j\}=\delta_{il},
\end{split}
\end{equation} 
where $\delta_{il}$ is the {\em Kronecker delta} operator, i.e., $\delta_{il}=0,\  \text{if} \  i\neq l$ and $\delta_{il}=1, \ \text{if} \ i=l$. 
We obtain the following equation by combining (\ref{eq:temp1}) and (\ref{eq:temp2}),
\begin{equation} \label{eq:temp3}
\begin{split}
T_1 &= \sum_{k=1}^m \bigg[\sum_{l=1}^n {\delta_{il}Q_k[l,j]G[l,j]\bigg]}\alpha_k[j] \\
&= \sum_{k=1}^m {Q_k[i,j]G_k[i,j]\alpha_k[j]}.
\end{split}
\end{equation}
Now, consider the second term in (\ref{eq:markov_sum}):
\begin{equation*}
\begin{split}
T_2&=\pr\{\sigma(t)=\sigma_{\off},s(t+1)=s_i|s(t)=s_j\}\\
&=\underbrace{\pr\{s(t+1)=i|s(t)=s_j,\sigma(t)=\sigma_{\off}\}}_{G_{\off}[i,j]}\\
& \hspace{2cm} \times \underbrace{\pr\{\sigma(t)=\sigma_{\off}|(s(t)=s_j\}}_{1-\pr\{\sigma(t)=\sigma_{on}|s(t)=s_j\}}.
\end{split}
\end{equation*}
Here, given the current state, the probability of being ON is sum of the probabilities of taking each action for the ON mode, i.e.,
\begin{equation*}
\begin{split}
\pr\{\sigma(t)& = \sigma_{on}|s(t)=s_j\} \\
&= \sum_{l=1}^n {\pr\{\sigma(t) = \sigma_{on}, s(t+1)=s_l|s(t)=s_j\}}\\
& =\sum_{k=1}^m \sum_{l=1}^n {Q_k[l,j]G_k[l,j]\alpha_k[j]},
\end{split}
\end{equation*}
which follows from (\ref{eq:temp3}).
Now, combining the expressions we get for $T_1$ and $T_2$, the equation (\ref{eq:markov_sum}) becomes:
\begin{equation*}
\begin{split}
M[i,j] =& \sum_{k=1}^m {Q_k[i,j]G_k[i,j]\alpha_k[j]}  \\
& +G_{\off}[i,j]\left(1-\sum_{k=1}^m \sum_{l=1}^n {Q_k[l,j]G_k[l,j]\alpha_k[j]}\right)
\end{split}
\end{equation*}
which can be written in compact matrix notation as follows
\begin{equation}
\begin{split}
M =& \sum_{k=1}^m  G_k  \odot Q_k \odot (\one\alpha_k^T) \,  \\
     &+\,  G_{\off}\! \odot\! \left(\one \left(\one^T   \!-\! \one^T   \sum_{k=1}^m  G_k \! \odot\! Q_k \!\odot\! (\one\alpha_k^T)  \right) \right)  .
\end{split}
\end{equation}
With the change of variables given in (\ref{eq:Pk}), the equation given above is equivalent to (\ref{eq:Mmat}).
%
\\
Finally, we will show that $M \in \mathbb{P}^{n \times n}$. For nonnegativity, we will consider $M$ in two terms. As both terms corresponds to probabilistic quantities, they both must be nonnegative. The nonnegativity of the first term, $\sum_{k=1}^m G_k \odot P_k$ is clear since $G_k \geq 0$ and $P_k \geq 0$, $k\in\Nm$. For the nonnegativity of the second term, i.e.,
$$G_{\off}\! \odot\! \left(\one \left(\one^T   \!-\! \one^T   \sum_{k=1}^m  G_k \! \odot\! P_k\right)\right)$$ 
we need to show
\begin{equation}\label{eq:temp4}
\one^T   \sum_{k=1}^m  G_k  \odot P_k \leq \one^T.
\end{equation}
Note that $\one^T   \sum_{k=1}^m  G_k  \odot P_k$ can be written in index notation as follows: For $j\in \Nm$,
\begin{equation}
\one^T   (\sum_{k=1}^m  G_k  \odot P_k )e_j = \sum_{i=1}^n \sum_{k=1}^m G_k[i,j]P_k[i,j].
\end{equation}
Here, $\sum_{k=1}^m P_k[i,j] \leq 1$  for any $i$ and $j$ since 
$$\sum_{k=1}^m P_k \leq \sum_{k=1}^m\one \alpha_k^T \leq \one\one^T,$$ 
and $\sum_{i=1}^n G_k[i,j] = 1$ for any $j$ and $k$ since $G_k$ is column stochastic.
Hence,
\begin{eqnarray*}
\sum_{i=1}^n \sum_{k=1}^m G_k[i,j]P_k[i,j] &=& \sum_{k=1}^m \underset{i \in \Nn}{\operatorname{conv}} P_k[i,j] \\
&\leq& \sum_{k=1}^m  \max_{i\in  \mathbb{N}^{+}_n} P_k[i,j] \\
&\ \leq& 1. 
\end{eqnarray*}
The last inequality follows from (\ref{eq:Pk_sat}), which then implies that
$$
\one^T   (\sum_{k=1}^m  G_k  \odot P_k )e_j \leq 1, \ j \in \Nn,
$$
and hence (\ref{eq:temp4}) holds. This concludes that $M~\geq~0$.

With $G_{\off} \in \mathbb{P}^{n\times n}$, the above inequality implies that $$ G_{\off} \odot \left(\one \left(\one^T   - \one^T   \sum_{k=1}^m  G_k  \odot P_k \right) \right) \geq \zero,$$ and hence $M \! \geq \! \zero$. Next let $H:=\sum_{k=1}^m  G_k  \odot P_k  $, then
\begin{eqnarray*}
\one^T M &=& \one^T H + \one^T G_\off  \odot (\one (\one^T -\one^T H) ) \\
&=& \one^T H + \one^T\left(G_\off - G_\off  \odot (\one \one^T H)\right) \\
&=& \one^T H + \one^T - \one^T \left( G_\off  \odot (\one \one^T H) \right).
\end{eqnarray*}
Let $\phi :=  H^T \one$,  then
\begin{eqnarray*}
\one^T M &=&\phi^T + \one^T - \one^T \left( G_\off  \odot (\one \phi^T) \right)\\
&=& \phi^T + \one^T - (\one^T G_\off)\odot \phi^T \\
&=& \phi^T + \one^T -\one^T \odot \phi^T = \phi^T + \one^T -\phi^T = \one^T,
\end{eqnarray*}
hence $M \in \mathbb{P}^{n \times n}$.
\end{proof}

Theorem~\ref{lem:Mdyn} shows that $M$ is a linear function of $\{P_k : k=1,\dots, m\}$.  This linearity property will be used for a convex synthesis of $P_k$ so that the matrix $M$ satisfies some favorable properties as convergence and safety as we will show later on in the paper.  The algorithm design parameters $\alpha_k$ and $Q_k$ can then be extracted from $P_k$ as we will show next. Thus the design of $\{P_k : k=1,\dots, m\}$ is an intermediary step to set the parameters of Algorithm~\ref{alg2}. 
\subsection{Extraction of $\alpha_k$ and $Q_k$ from $P_k$}

Once $P_k$ matrices are computed (which will explained in later sections), the selection of $\alpha_k$ and $Q_k$ can be done in multiple ways, that is, the same $P_k$ matrices can be parameterized in multiple ways via $\alpha_k$ and $Q_k$.  The choice must preserve the following conditions on $\alpha_k$ and $Q_k$ (since they contain probabilities of events as entries)
\begin{equation} \label{eq:Qalpha_props}
0 \leq Q_k \leq \one \one^T , \ 0\leq \alpha_k \leq \one , \, k\in \Nm , \quad \sum_{k=1}^m \alpha_k \leq \one. 
\end{equation}
Our default parameterization is:
\begin{equation} \label{eq:alpha_c} 
\alpha_k[j] = \max _{i\in \mathbb{N}_n^+} P_k[i,j], \ \ k \in \mathbb{N}_m^+,\ \ j\in \mathbb{N}_n^+.
\end{equation}
Hence the last inequality in (\ref{eq:Qalpha_props}) is satisfied (due to (\ref{eq:Pk_sat})), and we can choose $Q_k$ as 
\begin{equation} \label{eq:Q_c}
Q_k = P_k \ \dg(\alpha_k)^{-1},\ \ k \in \mathbb{N}_m^+ .
\end{equation}
Here, $Q_k$ is obtained by dividing each entry in a column of $P_k$ by the maximum element in that column, hence $0 \leq Q_k \leq \one \one^T$. For the same reason, note that any choice of $\alpha_k$ greater than or equal to the choices given in (\ref{eq:alpha_c}) would have resulted in feasible $Q_k$.
Also observe  that this particular choice $\alpha_k$ allows ``no action observed" cases, since it leads to $\sum_{k=1}^m \alpha_k \leq \one$ (the  sum does not have to be one). We can normalize $\alpha_k$'s such that an action is always observed, without changing the resulting $M$, as follows: Form a matrix $\Gamma$ with $\alpha_k$'s computed via (\ref{eq:alpha_c}) as its columns. Then compute a new set of $\alpha_k$'s by using the following expression and $Q_k$'s by using (\ref{eq:Q_c}),
\begin{equation}\label{eq:alpha_c2}
\Lambda:=\left[ \alpha_1 \ldots \alpha_m \right] = \dg(\Gamma \, \one)^{-1} \, \Gamma.
\end{equation}
Note that $\sum_{k=1}^m \alpha_k = \Lambda \one = \dg(\Gamma \one)^{-1} \Gamma \one = \one $. Since the second choice of $\alpha_k$ always produces values that are at least as large as the first choice in (\ref{eq:alpha_c}), we ensure that  $0 \leq Q_k \leq \one \one^T$.

\subsection{Connections to Markov Decision Processes }
Any designed matrices $\{P_k, \ k\in \mathbb{N}_m^{+}\}$ for the mode-switching agent model given in this paper defines a Markov chain for the system (\ref{eq:markov}) whose transition matrix $M$ is determined by Theorem~\ref{lem:Mdyn}. Hence, the model can be considered as a \emph{controlled Markov chain} model, and inherently there is a connection to Markov Decision Processes (MDPs). Notwithstanding this, the main objectives of the mode-switching agent is to shape the transition matrix $M$ to achieve a density distribution with some favorable properties such as \emph{safety} or a desired stationary distribution as we will discuss later in the paper. MDPs on the other hand, select policies that optimize a reward-based function. While we do not explicitly define a reward function, an optimized decision-making policy can be implicitly embedded into the system through the OFF mode.  Since MDP models in general do not observe the outcome of actions before transition, the  OFF mode can correspond to  a Markov chain resulting from an  MDP with a standard decision policy.  Hence, the  transition matrix $G_{\off}=M^{\pi}_{\rm MDP}$ would correspond to an MDP with a reward-optimized policy $\pi$. This default MDP policy may not satisfy some of the desired steady-state distribution or the safety constraints, for which we utilize the additional  observations in the ON mode via the synthesized decision policies  to achieve these design specifications.

\begin{figure}[ht]
\begin{center}
\includegraphics[width=2in]{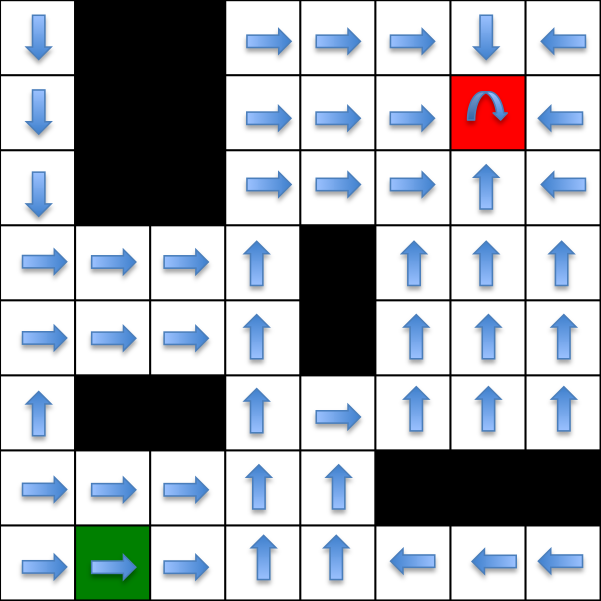}
\end{center}
\caption{An MDP optimal policy for deterministic motion planning where the objective is to go from the green bin to the red bin. Black squares are obstacles.}
\label{MDPgrid}
\end{figure}

To elaborate more on the connection with MDPs, we give in Figure \ref{MDPgrid} a simple motion planning example to show that the mode-switching agent model can implicitly incorporate a reward-optimizing MDP policy. The objective of the simplified motion planning example is to send an agent from  a source state (the green bin) to a destination state (the red bin) using the least number of transitions (i.e., moving along the shortest path). The arrows in the figure show \emph{an optimal MDP policy} to go from any state  to the destination state using a shortest path. This optimal policy can be obtained using the classical backward induction (dynamic programming) algorithm. Furthermore, in this example, environment is non-stochastic  and thus transitions due to the actions of the optimal policy are assumed deterministic. Note that if many agents start from the same green bin, they will all follow the same path. Thus, if bins are subject to varying capacity constraints (i.e., each bin has a maximum capacity for the number of agents that can be present in that bin at any time instant), then  the optimal MDP policy violates the constraints.  
The mode-switching agent model given in this paper can be used to handle such a situation as we argue next. Let the OFF mode be the transition matrix corresponding to the optimal MDP policy given in Figure \ref{MDPgrid}. Since the environmental transitions are deterministic, then so are the observations. Thus, $\{\alpha _k, k\in\mathbb{N}_{m}^+\}$ boils down to the probability of choosing an action deviating from the optimal policy to satisfy the capacity constraints. We can impose such capacity constraints directly on the resulting transition matrix of the system after substituting the relevant quantities of the problem in equation \eqref{eq:markov} to obtain a ``randomized'' policy that deviates from the optimal policy to satisfy the imposed constraints. These type of constraints are a  special type of \emph{safety} constraints that are discussed in more details further in the paper. It is worth noting that, in this example, both ON and OFF modes have the same set of actions, where the OFF mode executes a standard MDP policy (e.g., choosing actions for the shortest path) while the ON mode selects an action to observe  in order to deviate from the policy in the OFF mode. 


{\bf Remark:}  It is reasonable to consider a cost for having an observation causing the deviation from an optimal policy. However, having costs for observations and comparing the effect of deviation on the overall performance of the resulting MDP is an ongoing research direction and is not pursued further in this paper. $\hfill \blacksquare$
%
%
\section{Convex Synthesis of Safe Markov Chain for ON/OFF Agents}\label{sec:ConvexSynthesis}
We can  describe the decision policy design problem for ON/OFF agents as follows: Design $Q_k, \alpha _k, \ k \in \mathbb{N}^{+}_m$,  such that the density, as defined in (\ref{eq:pdf}), should satisfy the following constraints for $t \in \mathbb{N}$:
\begin{eqnarray}
&\nonumber \hspace{-6cm}\rm{Transition:}\\
 & \pr \{s(t\!+\!1) \!=\! s_i | s(t) \!=\!  s_i \} \!=\! 0 \ {\rm if}\  A_a[j,i] \!=\!  0, 
 \\
 &\nonumber \hspace{-6.5cm}\rm{Safety:}\\
 & Lx(t)  \leq  p , \ \ \forall x(t) \leq q,\label{eq:safety1}\\
 &\nonumber \hspace{-5.5cm}\rm{Convergence:}\\
  & \lim_{t \to \infty} x(t) = v, \ \ \forall x(0) \in \mathbb{P}^n,\label{eq:convergence}
\end{eqnarray}
where $v$  is a desired, discrete,  probability distribution,  and $A_a$ is the adjacency matrix defining the allowable transitions between states between two consecutive time steps, $L$, $q$, and $p$ are given matrices and arrays specifying safety constraints. In this section, we provide brief discussions on these constraints and express them as equivalent convex constraints on the Markov matrix $M$. Since the matrix $M$ depends linearly on the ON/OFF decision policy matrices $P_k$ as in (\ref{eq:Mmat}), this will imply that the constraints will be convex constraints on $P_k$'s, which are our design variables. Once $P_k$'s  are  computed, we can choose $\alpha_k$'s and $Q_k$'s by using (\ref{eq:alpha_c}) or (\ref{eq:alpha_c2}) and (\ref{eq:Q_c}).
\subsection*{Transition Constraints:} It is useful to impose constraints on the physically realizable state transitions by specifying some entries of the matrix $M$ as zeros. In the case of ON/OFF agents, these constraints are automatically ensured by stochastic transition matrices: If a state transition is simply not possible naturally, then the corresponding entry in the matrix $G_k, \ \forall k$ is zero, which implies that it is also zero in the matrix $M$, i.e., $M[i,j]=0$ if $G_k[i,j]=0, \forall k$.  If additional transitions are also needed to be eliminated due to mission specific reasons, we can impose additional constraints beyond the ones that are naturally imposed by matrix $G$ with the following constraint:
\begin{equation} \label{eq:transition}
(\one \one^T-A_a^T)\odot M = \zero
\end{equation}
where $A_a$ is the adjacency matrix, i.e., $A_a[i,j]=1$, if transition $i \rightarrow j$ is allowable and $A_a[i,j]=0$ otherwise.
Note that, this constraint is already linear, hence convex in $M$. 
\subsection*{Safety Constraints:}We consider a general form of linear density safety constraints: 
\begin{equation}\label{eq:safety}
L x(t)  \leq  p, \ \ \forall x(t) \leq q
\end{equation} which covers several types of safety constraints with different selections of $L$, $p$ and $q$. Two examples of these constraints are:
(i) Density upper bound constraints; (ii) Density rate constraints.
Density upper bound constraint ensures that the density of each state stays below a prescribed value,
that is,  
\begin{equation} \label{eq:dup}
x(t)\leq d, \quad t=1,2,\ldots
\end{equation} 
where $\zero\leq d \leq \one$ defines the density upper bounds in each state  and  it is assumed that   $x(0)\leq d$. Note that, this form can be obtained by letting $L=M$, $q=d$ and $p=d$ in the general form given in (\ref{eq:safety}). The density rate constraint is used to limit the rate of change of density for each state of Markov chain. \begin{equation} \label{eq:flowup}
-f\leq x(t+1)-x(t)\leq f, \quad t=0,1,\ldots,
\end{equation}  
where $f\geq \zero$ bounds the flow rate. Note that this is equivalent to case where 
$$L= \begin{bmatrix}
M-I\\
I-M
\end{bmatrix}, \  \  q= \begin{bmatrix}
f\\
f
\end{bmatrix} , 
$$ 
 and $p=d$ in (\ref{eq:safety}). For more examples on the linear safety constraints captured by the general form and for the proof of Lemma~\ref{thm:safetygeneral}, see \cite{auto2015}.

The safety constraints given in (\ref{eq:safety}) are ensured by the following lemma, which gives necessary and sufficient conditions for safety as linear inequalities on $M$.

\begin{lemma}\cite{swarm_denc15_tac,auto2015}\label{thm:safetygeneral}
Consider  the  Markov chain given by (\ref{eq:markov}).  Then, 
\begin{equation}\label{eq:gen_safety}
Lx(t) \!\leq\! q, \ \ \forall  x(t) \leq p,  
\end{equation} 
 {\em if and only if} 
there exist $S\!\in \!\Reals^{n\times n}$ and $y \! \in \!\Reals^n$ such that 
\begin{equation}\label{eq:dual_syn}
\begin{split}
S \! \geq \! \zero, \quad 
\left[L+S+y \one^T \right] \geq 0, \\
y+q \geq \left[L+S+y \one^T \right]p.
\end{split}
\end{equation}  
\end{lemma}
So far, we have obtained the linear equivalent conditions on Markov matrix for the transition, safety and ON/OFF constraints. 
 Hence, for the convex optimization problem, we define a set for feasible Markov matrices which satisfy safety, transition and ON/OFF constraints:
 $$
 \mathcal{M}_F = \{M \in \mathbb{P}^{n\times n} : \text{M satisfies  (\ref{eq:Mmat}), (\ref{eq:Pk_sat}), (\ref{eq:transition}), (\ref{eq:dual_syn}) \}}.
 $$

\vspace{2mm}
\fbox{
\begin{minipage}{3in}
\begin{equation} \label{syn_lmi1} 
\begin{split}
&M \in \mathcal{M}_F  \quad \text{\em if and only if:} \\
&  M \ge \zero, S\ge \zero, P_k \ge \zero \text{ for } k=1\dots m,\\
& \one^TM=\one^T, \\
& (\one \one^T -A_a^T) \odot M = \zero, \\
&[L+S+y\one^T] \geq 0, \\
& y+q \geq [L+S+y\one^T]p, \\
&M = \sum_{k=1}^m  G_k  \odot P_k  \, \\
&\hspace{0.8cm}+ \,  G_{\off} \odot \left(\one \left(\one^T   - \one^T   \sum_{k=1}^m  G_k  \odot P_k \right) \right), \\
&\sum_{k=1}^m {\max_{i\in  \mathbb{N}^{+}_n} P_k[i,j]} \leq 1, \quad j \in \mathbb{N}^{+}_n
 \end{split}
\end{equation}
\end{minipage}}\vspace{4mm}

Note that the last inequality above in (\ref{syn_lmi1}) can be written by using linear inequalities, hence it is a convex constraint. To see that define 
$$Z_j : = \left[ P_1(:,j) \  P_2(:,j) \ \ldots \  P_m(:,j) \right] , \ j=1,...,n,
$$ 
where $P_k(:,j)$ is the $j$'th column of $P_k$. Then we can replace the last inequality by the following inequalities
$$
Z_j \leq \one \beta_j^T, \ \ \one^T \beta_j \leq 1, \ \ \ j=1,...,n,
$$
where $\beta_j$'s are $m\times 1$ slack variables. 


\subsection*{Formulation of Convergence/Coverage Constraints:} As a part of the mission objectives for multi-agent systems, the agent distribution $x(t)$ is required to converge to a desired distribution $v \in \Probs^n$, i.e., 
\begin{equation}\label{eq:x2v}
 \tgi x(t) = v \ \   \forall \, x(0) \in \Probs^n.
\end{equation}
Since $M$ is a Markov matrix, hence column stochastic, a necessary condition for the desired convergence is that the desired distribution $v \in \Probs^n$ is an eigenvector of $M$: 
\begin{equation}
Mv =v.
\end{equation}
Notice that $Mv =v$ are simple linear equalities, and hence they are convex constraints but are not sufficient for convergence. In order to obtain the convex equivalent of the ergodicity constraint ($\tgi x(t) = v \ \   \forall \, x(0) \in \Probs^n$), there are three possible approaches. The first approach is to use the results in \cite{ajc_13} which developed  a sufficient condition  for ergodicity based on the LMI characterizations of the stability of discrete-time systems  (by leveraging results from  \cite{oliveria_bmi}).  The sufficient condition is : There exist $P=P^T\succ \zero$, $\lambda \in [0,1)$, and $F$ such  that 
\begin{equation}
 \begin{bmatrix}
\lambda^2P & (M-v\one^T)^TF^T \\
F(M-v\one^T) & F+F^T-P
\end{bmatrix} \succeq 0.
\label{eq:min_prob}
\end{equation}
The above condition is an LMI for prescribed  values of $\lambda$ and $F$ which we typically choose to be $F=\dg{(v)}^{-1}$. A more detailed discussion on the selection of $F$ and $\lambda$ can be found in \cite{swarm_denc15_tac}. Here, $\lambda$ denotes the exponential decay rate of the error, $e(t)=x(t)-v$ and can be computed  via a line search (by noting that the above inequalities are LMIs for a given $\lambda$ and $F=\dg{(v)}^{-1}$) to achieve the largest feasible convergence rate.  

As a second approach, we can impose {\em reversibility} on the Markov matrix.   Then, if $v\!>\! \zero$, we can state a necessary and sufficient condition \cite{boyd_09} for ergodicity constraint as  follows:
\begin{equation} \label{eq:reversible_LMI}
-\lambda I  
\preceq H^{-1} M H -hh^T \preceq \lambda I  ,
\end{equation}
where  $h=(v_1^{1/2},\ldots,v_m^{1/2})$ and $H= {\rm\dg}(h)$. Note that $\lambda$, the convergence rate, can be minimized for fastest convergence within a convex problem formulation.

Finally, the third approach is to impose strong connectivity on the corresponding adjacency matrix of the Markov matrix, i.e. $\mI(M) \ \text{is strongly connected}$. This can be obtained using the following linear equations:
\begin{equation}\label{eq:M_connect}
M[i,j]\geq \epsilon \text{ if } A_a[i,j]=1,
\end{equation}
where $\epsilon>0$ is a small positive scalar (e.g., it can be set as the machine precision). This provides a sufficient condition for ergodicity of the Markov chain. Thus if $Mv=v$ and the adjacency matrix $A_a$ corresponds to a connected graph, then $v$ is unique stationary distribution and $\tgi x(t) = v \ \   \forall \, x(0) \in \Probs^n$.   Note that, though we can ensure convergence with simpler inequalities, we cannot impose  or minimize convergence rate directly in this approach.

\subsection*{Formulation of synthesis as an optimization problem:}
We can formulate the Markov matrix synthesis as a minimization problem with the desired constraints. One example cost function is $\one^T(\one - \dg(M))$ which aims to minimize overall action, i.e., $M\simeq I$:\\
\begin{equation} \label{syn_lmi2} 
\begin{split}
\min_{M} & \quad \one^T(\one - \dg(M)) \quad \text{\em such that}\\
&M \in \mathcal{M}_F \\
& (\ref{eq:min_prob}) \ \text{or} \ (\ref{eq:reversible_LMI}) \ \text{or} \ (\ref{eq:M_connect}).
\end{split}
\end{equation}

\section{Numerical Example}\label{sec:simulation}
This section presents an  illustrative numerical example for  the   density control problem for autonomous agents  with ON/OFF control modes. We consider  a swarm of mobile agents that  are distributed over the configuration space (see Figure \ref{Sim}) that is partitioned to 8 subregions, which are referred to  as {\em bins}. In this configuration, probabilistic density distribution is given as $x[i](t):= \pr \{r(t) \in R_i\}$ where $r(t)$ is the position vector of an agent at time step $t$. For this setting, safety upper bound constraint is used to limit the expected number of agents in each bin. Two sets of simulations are performed by using the ON/OFF policy synthesized by solving the LMI problem in (\ref{syn_lmi2}) with (\ref{eq:min_prob}), both with the density upper bound constraints and $m=5$ actions for the ON mode: (i) Total $N_{s1} = 3000$ simulations with same safe initial condition, i.e., same $x(0)$ with different realizations; (ii) Total $N_{s2}=3000$ simulations with randomly generated $3000$ safe initial conditions. For all cases, OFF case corresponds to ``no action", i.e., $G_{\off} = I$. For the actions in the ON mode, column stochastic $G_k$ matrices are selected such  that they have different steady-state final distributions and do not satisfy safety and transition constraints. 

Other parameters for the simulations are set as follows:\vspace{3mm}
%
\begin{itemize}\itemsep 7pt
 \item[] $A_a =$ {\footnotesize $\begin{bmatrix}1 & 1 & 1 & 0 & 0 & 0 & 0 & 0\\ 1 & 1 & 1 & 1 & 1 & 0 & 0 & 0\\ 1 & 1 & 1 & 0 & 0 & 0 & 0 & 0\\ 0 & 1 & 0 & 1 & 1 & 0 & 1 & 0\\ 0 & 1 & 0 & 1 & 1 & 0 & 1 & 0\\ 0 & 0 & 0 & 0 & 0 & 1 & 1 & 1\\ 0 & 0 & 0 & 1 & 1 & 1 & 1 & 1\\ 0 & 0 & 0 & 0 & 0 & 1 & 1 & 1 \end{bmatrix}$} 
 \item[] $N_a=3000$
 \item[] $x(0)=[0.5 \ \ 0 \ \ 0.5 \ \ 0 \ \ 0 \ \ 0 \ \ 0 \ \ 0]^T$
 \item[] $v=[0.005 \ \ 0.02 \ \ 0.005 \ \ 0.04 \ \ 0.05 \ \ 0.34 \ \ 0.2 \ \ 0.34]^T$
 \item[] $d=[1 \ \ 0.15 \ \ 1 \ \ 0.12 \ \ 0.12 \ \ 1 \ \ 0.4 \ \ 1]^T $
 \item[] $\lambda =0.975$
\end{itemize}
\noindent where $A_a$ is the adjacency matrix of the bin connections, $N_a$ is the total number of agents, $x(0)$ is the initial distribution of agents, $v$ is the desired final distribution as in equation \eqref{eq:convergence}, $d$ is a safety upper bound constraint (as in equation \eqref{eq:safety} with $L=I$ and $p=q=d$), and $\lambda$ is the convergence rate of the system.

\begin{figure}[t!]
\begin{center}
\includegraphics[width=3.5in]{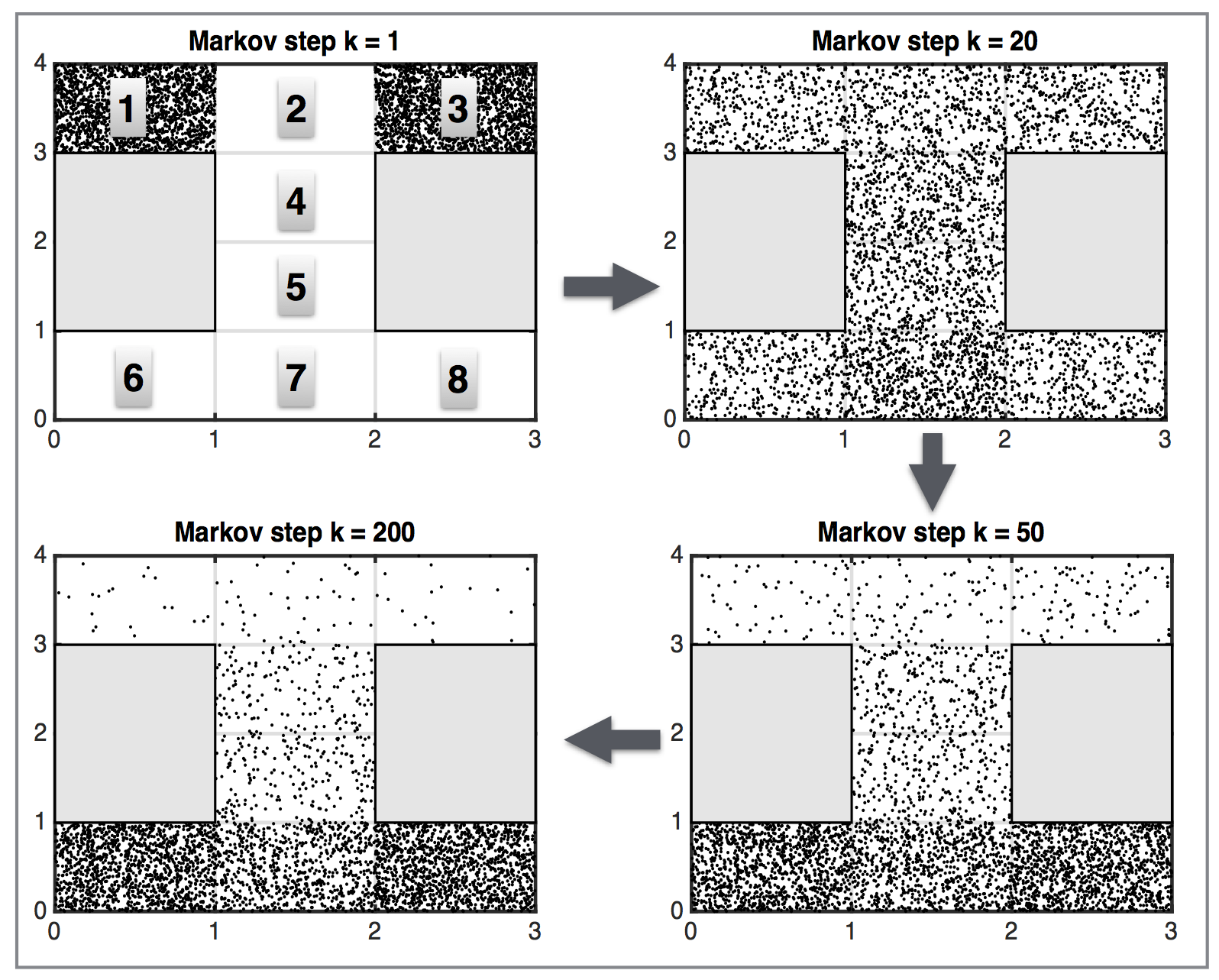}
\end{center}
\vspace{-5mm}
\caption{Snapshots of simulation: configuration space with bin numbers}
\label{Sim}
\end{figure}
\vspace{-1mm}

Since the bins at the corners behave like accumulation points in the initial and final distributions, the safety upper bound constraints are not imposed for these bins, i.e., the corresponding entries of the $d$ vector are set as $1$. With the given parameters, the optimization problem given in (\ref{syn_lmi1})  with (\ref{eq:min_prob}) is solved using YALMIP and SDPT3 
\cite{yalmip,tutuncu03}. $G_k$ matrices used in the simulations and the final distributions, $v_k$, corresponding to each $G_k$ are given in the appendix along with the resulting solution variables $Q_k$ and $\alpha_k$.
In many applications, environmental transition matrices, $G_k$'s, may satisfy transition constraints in most examples, i.e., $G_k[i,j] = 0$ when $A_a[j,i]=0$. However this example considers some  $G_k$ matrices that do not have this property, that is, we do not allow some motions even when they can be induced by the environment. For example, as given in the appendix, $G_1[4,1]>0$ even though $A_a[1,4]=0$ and hence $Q_1[4,1]=0$.  This makes the problem more challenging  since motion constraints are not automatically satisfied by the environment and they must be ensured by the policy. 
Such scenarios can arise in the balloon motion control example given in the introduction \cite{elfes2010implications,kuwata2009decomposition,wolf2010probabilistic}
 where environmental transition matrices may not satisfy the desired constraints. Though some altitudes may induce high velocities, we may not choose to ride with such fast winds, for example, not to damage the structural integrity of the balloon (there may be a maximum speed limit for structural safety purposes).

Simulation results are presented in Figure \ref{Density}. The mean density $\bf{x}$ and $3\sigma$ confidence bounds are shown for the case with density upper bound $d$. The average density for the case without constraints is obtained by evolving the density according to equation (\ref{eq:markov}). Density goes above the desired upper bound for the bins $2,4$ and $5$ when the constraint is not imposed. By using ON/OFF control policy, we are able to ensure that the density does not go beyond the prescribed upper limit at a reasonable cost of reduced convergence rate.

For the results of the second set of simulations, point-wise maximum values of the density at each time step over all $3000$ simulations are plotted. This is a good demonstration of our claim in Lemma~\ref{thm:safetygeneral}: For all safe initial conditions, i.e., $x(0) \le d$, the density is guaranteed to satisfy safety constraints for all time, i.e., $x(t) \le d, \ t\in \mathbb{N}^+$.

The snapshots of the overall distribution taken at the beginning and at the end of a simulation from the first set are shown in Figure \ref{Sim} which also shows the bin numbers. 

\begin{center}
\begin{figure*}[ht!]
\begin{minipage}{7.2in}
\includegraphics[width=7.2in,height=4in]{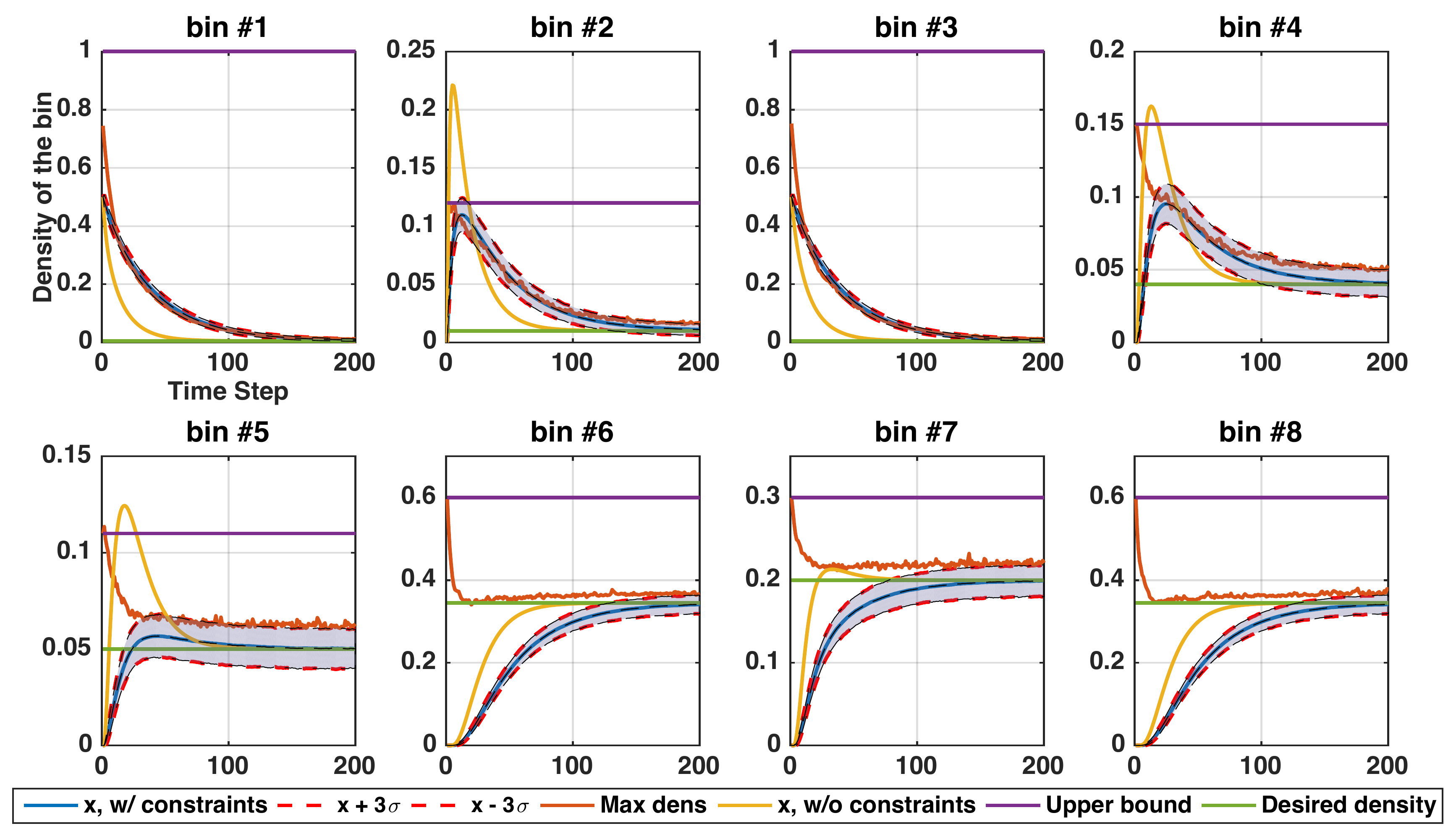}
\vspace{-5mm}
\caption{Time history of the density of each bin. With the density upper bound constraints the density is guaranteed to stay in the tube between red dashed lines with 99.7\% confidence.}
\label{Density}
\end{minipage}
\end{figure*}
\end{center}
\vspace{-5mm}

%
\section{Conclusion}\label{sec:conclusion}
In this paper, we develop a probabilistic density control policy for autonomous mobile agents with two modes: ON and OFF. When the agent is in the ON mode, it can observe the one-step outcome of a single action chosen from actions for the ON mode and can decide whether to take this action or not. If it does not take the action, it switches to the OFF mode. The density distribution of agents in the system evolves according to a Markov chain that, as shown in this paper, is a linear function of the stochastic environment and the decision policy. We formulate a convex optimization problem, that can be solved reliably via interior-point methods, to synthesize the decision policy which ensures desired safety, transition and convergence properties for the underlying Markov chain. The given constraints on the density are equivalently expressed as constraints on the Markov chain. The  resulting  density control model is illustrated with a numerical example on autonomous mobile agents. 

\bibliographystyle{IEEEtran}
 \bibliography{swarm.bib}

\begin{appendix}
The problem parameters and  the resulting solution variables for the numerical example are: 
$$
G_1 \!=\!{\footnotesize 10^{-4}\!\! \left[
\arraycolsep=2.7pt
\begin{array}{cccccccc} 6735 & 384 & 284 & 146 & 179 & 27 & 34 & 22\\ 773 & 5252 & 290 & 339 & 337 & 24 & 42 & 22\\ 680 & 1080 & 7611 & 245 & 141 & 116 & 135 & 83\\ 505 & 1196 & 226 & 6428 & 278 & 95 & 436 & 122\\ 202 & 891 & 389 & 613 & 6425 & 79 & 421 & 88\\ 391 & 308 & 585 & 450 & 546 & 8012 & 506 & 1264\\ 349 & 540 & 157 & 1339 & 1696 & 666 & 7437 & 379\\ 365 & 349 & 458 & 440 & 398 & 981 & 989 & 8020 \end{array}\right]}
$$

$$
G_2 \!=\!{\footnotesize 10^{-4}\!\! \left[
\arraycolsep=2.7pt
\begin{array}{cccccccc} 1871 & 2359 & 1563 & 1590 & 1349 & 2650 & 2471 & 235\\ 1616 & 82 & 1651 & 2019 & 1008 & 56 & 1159 & 1820\\ 1269 & 194 & 283 & 1387 & 1117 & 924 & 2903 & 1992\\ 1314 & 2350 & 2205 & 149 & 1586 & 522 & 402 & 1211\\ 1209 & 695 & 1657 & 1766 & 1792 & 1928 & 101 & 1641\\ 824 & 614 & 1911 & 1827 & 511 & 2070 & 1478 & 1032\\ 1463 & 1806 & 155 & 778 & 1038 & 111 & 992 & 453\\ 434 & 1900 & 575 & 484 & 1599 & 1739 & 494 & 1616 \end{array}\right]}
$$
$$
G_3 \!=\!{\footnotesize 10^{-4}\!\! \left[
\arraycolsep=2.7pt
\begin{array}{cccccccc} 4579 & 389 & 324 & 291 & 360 & 53 & 69 & 46\\ 975 & 3954 & 359 & 293 & 320 & 47 & 84 & 43\\ 822 & 1304 & 5688 & 490 & 282 & 233 & 270 & 167\\ 1009 & 1171 & 451 & 5333 & 269 & 190 & 336 & 244\\ 404 & 787 & 777 & 750 & 5460 & 158 & 301 & 175\\ 783 & 617 & 1171 & 901 & 1092 & 7854 & 1012 & 697\\ 699 & 1081 & 314 & 1062 & 1422 & 878 & 7083 & 432\\ 729 & 697 & 916 & 880 & 795 & 587 & 845 & 8196 \end{array}\right]}
$$
$$
G_4 \!=\!{\footnotesize 10^{-4}\!\! \left[
\arraycolsep=2.7pt
\begin{array}{cccccccc} 5864 & 878 & 526 & 686 & 679 & 764 & 70 & 534\\ 780 & 4938 & 783 & 671 & 516 & 768 & 799 & 744\\ 571 & 906 & 4983 & 689 & 638 & 1027 & 761 & 424\\ 795 & 1231 & 693 & 5033 & 461 & 124 & 849 & 814\\ 687 & 771 & 593 & 925 & 4788 & 751 & 873 & 612\\ 744 & 560 & 836 & 665 & 598 & 4581 & 951 & 1066\\ 388 & 276 & 572 & 804 & 1530 & 807 & 4885 & 738\\ 171 & 440 & 1014 & 527 & 790 & 1178 & 812 & 5068 \end{array}\right]}
$$

$$
G_5 \!=\!{\footnotesize 10^{-4}\!\! \left[
\arraycolsep=2.7pt
\begin{array}{cccccccc} 8890 & 377 & 244 & 0 & 0 & 0 & 0 & 0\\ 572 & 6550 & 221 & 385 & 354 & 0 & 0 & 0\\ 538 & 856 & 9535 & 0 & 0 & 0 & 0 & 0\\ 0 & 1221 & 0 & 7522 & 287 & 0 & 535 & 0\\ 0 & 996 & 0 & 477 & 7390 & 0 & 540 & 0\\ 0 & 0 & 0 & 0 & 0 & 8170 & 0 & 1830\\ 0 & 0 & 0 & 1616 & 1969 & 455 & 7792 & 326\\ 0 & 0 & 0 & 0 & 0 & 1375 & 1133 & 7844 \end{array}\right]}
$$

$$v_1 = {\footnotesize \begin{bmatrix}   0.0200\\
    0.0200\\
    0.0600\\
    0.0600\\
    0.0600\\
    0.2900\\
    0.2000\\
    0.2900\\
    \end{bmatrix}},
 v_2 = {\footnotesize \begin{bmatrix}      0.1756\\
    0.1191\\
    0.1197\\
    0.1246\\
    0.1391\\
    0.1246\\
    0.0883\\
    0.1090 \end{bmatrix}},
 v_3 = {\footnotesize \begin{bmatrix}   0.0200\\
    0.0200\\
    0.0600\\
    0.0600\\
    0.0600\\
    0.2900\\
    0.2000\\
    0.2900\\
    \end{bmatrix}},
    $$
    $$
  v_4 ={\footnotesize \begin{bmatrix} 0.1250 \\
    0.1250\\
    0.1250\\
    0.1250\\
    0.1250\\
    0.1250\\
    0.1250\\
    0.1250 \end{bmatrix}},
 v_5 = {\footnotesize \begin{bmatrix}   0.0200\\
    0.0200\\
    0.0600\\
    0.0600\\
    0.0600\\
    0.2900\\
    0.2000\\
    0.2900\\
    \end{bmatrix}}.
$$

$$
\alpha_1 = {\footnotesize \begin{bmatrix}
    0.0940\\
    0.1139\\
    0.0568\\
    0.2006\\
    0.1975\\
    0.1656\\
    0.1573\\
    0.1477\end{bmatrix}},
\alpha_2 = {\footnotesize \begin{bmatrix}
    0.5594\\
    0.5165\\
    0.6630\\
    0.1332\\
    0.1125\\
    0.1229\\
    0.2167\\
    0.1553 \end{bmatrix}},
\alpha_3 = {\footnotesize\begin{bmatrix}
    0.1201\\
    0.1036\\
    0.0616\\
    0.1457\\
    0.1492\\
    0.1846\\
    0.2036\\
    0.1488 \end{bmatrix}},
$$
$$
\alpha_4 = {\footnotesize\begin{bmatrix}
    0.0941\\
    0.0987\\
    0.0969\\
    0.1135\\
    0.1676\\
    0.1845\\
    0.2150\\
    0.1955 \end{bmatrix}},
\alpha_5 = {\footnotesize \begin{bmatrix}
    0.0976\\
    0.1464\\
    0.0679\\
    0.3711\\
    0.3404\\
    0.1896\\
    0.1537\\
    0.1903\end{bmatrix}}
$$

$$
Q_1\!\! = \!\!{\scriptsize 10^{-4}\left[
\arraycolsep=2pt
\begin{array}{cccccccc} 6184 & 6205 & 7025 & 0 & 0 & 0 & 0 & 0\\ 10000 & 5983 & 10000 & 4886 & 7151 & 0 & 0 & 0\\ 7165 & 9178 & 7084 & 0 & 0 & 0 & 0 & 0\\ 0 & 10000 & 0 & 5888 & 4599 & 0 & 6142 & 0\\ 0 & 9932 & 0 & 8553 & 5895 & 0 & 7780 & 0\\ 0 & 0 & 0 & 0 & 0 & 6645 & 9431 & 8970\\ 0 & 0 & 0 & 10000 & 10000 & 10000 & 6141 & 10000\\ 0 & 0 & 0 & 0 & 0 & 8186 & 10000 & 7044 \end{array}\right]}
$$

$$
Q_2\!\! = \!\!{\scriptsize 10^{-4}\left[
\arraycolsep=2pt
\begin{array}{cccccccc} 5788 & 4625 & 5634 & 0 & 0 & 0 & 0 & 0\\ 10000 & 5780 & 10000 & 2621 & 8247 & 0 & 0 & 0\\ 9130 & 8424 & 5778 & 0 & 0 & 0 & 0 & 0\\ 0 & 10000 & 0 & 6133 & 2381 & 0 & 5928 & 0\\ 0 & 9913 & 0 & 9595 & 6170 & 0 & 6557 & 0\\ 0 & 0 & 0 & 0 & 0 & 7643 & 10000 & 8495\\ 0 & 0 & 0 & 10000 & 10000 & 8729 & 5970 & 10000\\ 0 & 0 & 0 & 0 & 0 & 10000 & 9441 & 6867 \end{array}\right]}
$$

$$
Q_3 \!\! = \!\! {\scriptsize
10^{-4}\left[
\arraycolsep=2pt
\begin{array}{cccccccc} 6037 & 6198 & 6769 & 0 & 0 & 0 & 0 & 0\\ 10000 & 6009 & 10000 & 5376 & 7016 & 0 & 0 & 0\\ 7375 & 9295 & 6838 & 0 & 0 & 0 & 0 & 0\\ 0 & 10000 & 0 & 6006 & 4978 & 0 & 6045 & 0\\ 0 & 9875 & 0 & 8594 & 5989 & 0 & 7564 & 0\\ 0 & 0 & 0 & 0 & 0 & 6406 & 10000 & 8110\\ 0 & 0 & 0 & 10000 & 10000 & 10000 & 6069 & 10000\\ 0 & 0 & 0 & 0 & 0 & 7443 & 9937 & 6929 \end{array}\right]}
$$

$$
Q_4\!\! = \!\!{\scriptsize 10^{-4}\left[
\arraycolsep=2pt
\begin{array}{cccccccc} 6181 & 6542 & 6026 & 0 & 0 & 0 & 0 & 0\\ 10000 & 6010 & 10000 & 4975 & 7535 & 0 & 0 & 0\\ 7022 & 8884 & 6154 & 0 & 0 & 0 & 0 & 0\\ 0 & 10000 & 0 & 6170 & 4073 & 0 & 5961 & 0\\ 0 & 9840 & 0 & 8807 & 5947 & 0 & 8956 & 0\\ 0 & 0 & 0 & 0 & 0 & 6455 & 10000 & 8130\\ 0 & 0 & 0 & 10000 & 10000 & 10000 & 6073 & 10000\\ 0 & 0 & 0 & 0 & 0 & 8266 & 9948 & 6375 \end{array}\right]}
$$

$$
Q_5 \!\! = \!\! {\scriptsize 10^{-4}\left[
\arraycolsep=2pt
\begin{array}{cccccccc} 
6326 & 6227 & 7107  & 0    & 0     & 0    & 0    & 0   \\ 10000 & 5957 & 10000 & 3764 & 7591 & 0    & 0    & 0   \\ 7150 & 9133 & 7195 & 0    & 0      & 0    & 0    & 0   \\ 
0    & 10000 & 0   & 5857 & 3727   & 0    & 6123 & 0   \\ 
0    & 9974 & 0    & 8794 & 5868   & 0    & 8017 & 0   \\ 
0    & 0    & 0    & 0    & 0      & 6913 & 6109 & 9742\\ 
0    & 0    & 0    & 10000 & 10000 & 10000 & 6134 & 10000\\ 
0    & 0    & 0    & 0     & 0     & 9090 & 10000 & 7037 \end{array}\right]}
$$

\end{appendix}


\begin{IEEEbiography}
[{\includegraphics[width=1in,height=1.25in,clip,keepaspectratio]{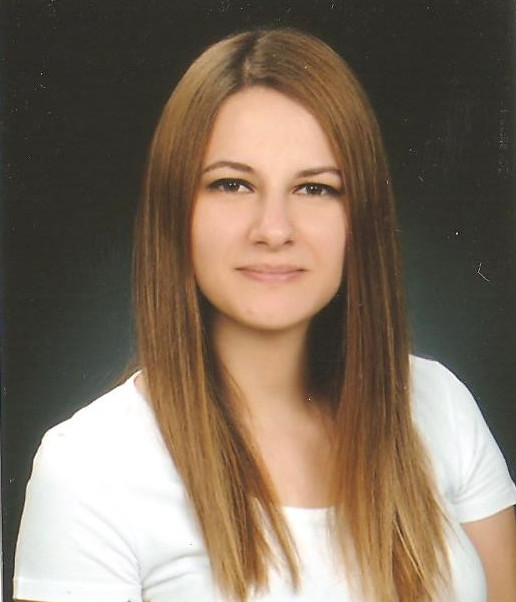}}]{Nazl\i \ Demirer}
is a Graduate Research Assistant with the Autonomous Control Laboratory in the William E. Boeing Department of Aeronautics and Astronautics at the University of Washington, USA.  She received the B.S. degree in Aerospace Engineering from the Middle East Technical University, Ankara, Turkey in 2012 and M.S degree from The University of Texas at Austin, TX in 2014. She is currently pursuing Ph.D. degree in Aeronautics and Astronautics from the University of Washington. Her current research interests include autonomous multi-agent systems and Markov chains.
\end{IEEEbiography} 
\begin{IEEEbiography}
[{\includegraphics[width=1in,height=1.25in,clip,keepaspectratio]{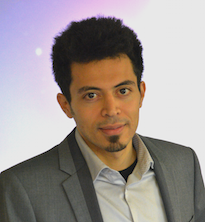}}]{Mahmoud El Chamie}
is a Research Associate with the Autonomous Control Laboratory in the William E. Boeing Department of Aeronautics and Astronautics at the University of Washington, USA.  He  received a M.S. (UBINET program) and a Ph.D. in computer science from the University of Nice Sophia Antipolis, France, in 2011 and 2014 respectively. His Ph.D. was held at Inria Sophia Antipolis with the Maestro team. His current research interests include optimization and control of autonomous multi-agent systems.
\end{IEEEbiography} 
 \begin{IEEEbiography} 
[{\includegraphics[width=1in,height=1.25in,clip,keepaspectratio]{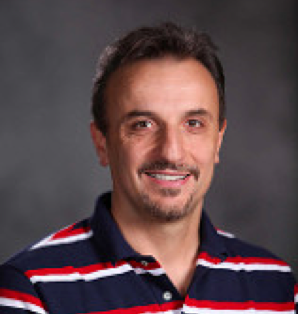}}]{Beh\c{c}et A\c{c}{\i}kme\c{s}e}
is an Associate Professor in the Department of Aeronautics and Astronautics at the University of Washington, Seattle. He received his Ph.D. from Purdue University and was a Visiting Assistant Professor at Purdue University before joining NASA Jet Propulsion Laboratory (JPL) in 2003. He was a senior technologist at JPL and a lecturer at Caltech, where he developed control algorithms for planetary landing, formation flying spacecraft, and asteroid and comet sample return missions. He is the developer of the “flyaway” control algorithms used in Mars Science Laboratory, which landed on Mars in August, 2012. He has joined the faculty of University of Texas at Austin in 2012, where he stayed until 2015.
\end{IEEEbiography}
\end{document}